\documentclass[a4paper]{amsart}

\usepackage{amssymb, enumerate}
\usepackage{aliascnt}
\usepackage{hyperref}
\usepackage{tikz}

\newtheorem{lma}{Lemma}[section]

\newaliascnt{thmCt}{lma}
\newtheorem{thm}[thmCt]{Theorem}
\aliascntresetthe{thmCt}

\newaliascnt{corCt}{lma}
\newtheorem{cor}[corCt]{Corollary}
\aliascntresetthe{corCt}

\newaliascnt{prpCt}{lma}
\newtheorem{prp}[prpCt]{Proposition}
\aliascntresetthe{prpCt}

\theoremstyle{definition}

\newaliascnt{pgrCt}{lma}

\aliascntresetthe{pgrCt}

\newaliascnt{dfnCt}{lma}
\newtheorem{dfn}[dfnCt]{Definition}
\aliascntresetthe{dfnCt}

\newaliascnt{rmkCt}{lma}
\newtheorem{rmk}[rmkCt]{Remark}
\aliascntresetthe{rmkCt}

\newaliascnt{rmksCt}{lma}
\newtheorem{rmks}[rmksCt]{Remarks}
\aliascntresetthe{rmksCt}

\newaliascnt{qstCt}{lma}
\newtheorem{qst}[qstCt]{Question}
\aliascntresetthe{qstCt}

\newaliascnt{ntnCt}{lma}

\aliascntresetthe{ntnCt}

\newcounter{theoremintro}

\newcommand{\NN}{\mathbb{N}}
\newcommand{\RR}{\mathbb{R}}
\newcommand{\ZZ}{\mathbb{Z}}
\newcommand{\CC}{{\mathbb C}}
\newcommand{\andSep}{\quad\text{ and }\quad}
\newcommand{\SubSep}{\mathrm{Sub}_{\mathrm{sep}}}
\newcommand{\sa}{\mathrm{sa}}
\newcommand{\grPre}{\gr_0}
\newcommand{\ca}{$C^*$-algebra}

\DeclareMathOperator{\diag}{diag}

\DeclareMathOperator{\rr}{rr}
\DeclareMathOperator{\gr}{gr}

\title{Generators in $\mathcal{Z}$-stable $C^*$-algebras of real rank zero}
\author{Hannes Thiel}
\address{Hannes Thiel
Mathematisches Institut, Universit\"at M\"unster, Einsteinstr.~62, 48149 M\"unster, Germany.}
\email{hannes.thiel@posteo.de}
\urladdr{http://hannesthiel.org}

\thanks{The author was partially supported by the Deutsche Forschungsgemeinschaft (DFG, German Research Foundation) under Germany's Excellence Strategy EXC 2044-390685587 (Mathematics M\"{u}nster: Dynamics-Geometry-Structure).
}
\subjclass[2010]%
{Primary
46L05; 
Secondary
46L85. 
}
\date{\today}

\begin{document}

\begin{abstract}
We show that every separable \ca{} of real rank zero that tensorially absorbs the Jiang-Su algebra contains a dense set of generators.

It follows that in every classifiable, simple, nuclear \ca{}, a generic element is a generator.
\end{abstract}

\maketitle

\section{Introduction}

The generator problem for \ca{s} asks to determine the minimal number of generators for a given \ca.
One difficulty when studying this problem is that the minimal number of generators is a rather ill-behaved invariant;
in particular, it may increase when passing to ideals or inductive limits.
We showed in \cite{Thi12arX:GenRnk} that a much better behaved invariant is obtained by considering the minimal number $n$ such that generating $n$-tuples are dense.

The most interesting aspect of the generator problem is to determine which \ca{s} are generated by a single operator, and it is a mayor open question if every separable, simple \ca{} is singly generated.
A natural variant of this problem is to describe the class of \ca{s} that contain a dense set of generators.
In previous work, we showed that this class includes all separable AF-algebras (\cite[Theorem~7.3]{Thi12arX:GenRnk}) and all separable, approximately subhomogeneous (ASH) algebras that are $\mathcal{Z}$-stable, that is, that tensorially absorb the Jiang-Su algebra $\mathcal{Z}$ (\cite[Theorem~5.10]{Thi20arX:grSubhom}).

In this paper, we show that every separable, $\mathcal{Z}$-stable \ca{} of real rank zero contains a dense set of generators;
see \autoref{prp:mainThm}.
This includes all separable, nuclear, purely infinite \ca{s} of real rank zero (\autoref{prp:nuclPI}) and in particular every Kirchberg algebra (\autoref{prp:KirchbergAlg}).
Together with the previous result about  $\mathcal{Z}$-stable ASH-algebras, we deduce that every classifiable, simple, nuclear \ca{} contains a dense set of generators;
see \autoref{prp:classifiable-gr1}.
\\

The main tool to prove these results is the generator rank, which was introduced in \cite{Thi12arX:GenRnk}.
A unital, separable \ca{} $A$ has generator rank at most $n$, denoted $\gr(A)\leq n$, if the set of self-adjoint $(n+1)$-tuples that generate $A$ as a \ca{} is dense;
see \autoref{dfn:gr} for the general definition in the nonunital and nonseparable case.
By \cite[Proposition~5.6]{Thi12arX:GenRnk}, $A$ has generator rank zero if and only if $A$ is commutative with totally disconnected spectrum.
Since two self-adjoint elements $a$ and $b$ generate the same sub-\ca{} as the element $a+ib$, it follows that $A$ has generator rank at most one if and only if $A$ contains a dense set of generators.

Our main result (\autoref{prp:mainThm}) shows that separable, $\mathcal{Z}$-stable \ca{s} of real rank zero have generator rank one.
To prove this, we heavily rely on the permanence properties of the generator rank that were established in \cite{Thi12arX:GenRnk}.
In particular, the generator rank does not increase when passing to ideals, quotients, or inductive limits, and we can estimate the generator rank of extensions;
see \autoref{prp:permanence}.

The general strategy is as follows:
Given a unital, separable, $\mathcal{Z}$-stable \ca{} $A$ of real rank zero, 
we use that $\mathcal{Z}$ is the inductive limit of generalized dimension-drop algebras $Z_{2^\infty,3^\infty}$ (see the proof of \autoref{prp:gr8} for the definition) to reduce the problem to computing the generator rank of $A\otimes Z_{2^\infty,3^\infty}$.

We obtain the estimate $\gr(A\otimes Z_{2^\infty,3^\infty})\leq 1$ by combining two results:
First, we show that $A\otimes Z_{2^\infty,3^\infty}$ has finite generator rank (we verify that it is at most $8$);
see \autoref{prp:gr8}.
Then, using a delicate construction of generators in $A\otimes Z_{2^\infty,3^\infty}$, we show that $\gr(A\otimes Z_{2^\infty,3^\infty})\leq n+1$ implies that $\gr(A\otimes Z_{2^\infty,3^\infty})\leq n$, for every $n\geq 1$;
see \autoref{prp:construction}.

To verify that $A\otimes Z_{2^\infty,3^\infty}$ has finite generator rank, we use the short exact sequence
\[
0 \to A\otimes C_0((0,1),M_{6^\infty}) \to A\otimes Z_{2^\infty,3^\infty} \to A\otimes(M_{2^\infty}\oplus M_{3^\infty}) \to 0,
\]
and it suffices to show that the ideal and the quotient have finite generator rank.
To estimate the generator rank of $A\otimes C_0((0,1),M_{6^\infty})$, we use that it is an ideal in $A\otimes C_0([0,1],M_{6^\infty})$.
In \autoref{sec:tensI}, we prove an upper bound for the generator rank of such algebras.

The algebra $A\otimes Z_{2^\infty,3^\infty}$ has a natural structure as a $C([0,1])$-algebra, with each fiber over $(0,1)$ isomorphic to $A\otimes M_{6^\infty}$, and with the fibers at $0$ and $1$ isomorphic to $A\otimes M_{3^\infty}$ and $A\otimes M_{2^\infty}$, respectively.
We use a Stone-Weierstra{\ss}-type result that characterizes when a set of self-adjoint elements generates $A\otimes Z_{2^\infty,3^\infty}$:
the elements have to generate each fiber, and they have to suitably separate the points in $[0,1]$.
The assumption of real rank zero is crucial for both:
it is used in the proof of \autoref{prp:UHFabs} to show that the tensor product of $A$ with a UHF-algebra has generator rank one;
and it is used in the proofs of \autoref{prp:approxSepInterval} and \autoref{prp:construction} to construct self-adjoint elements that separate the points in $[0,1]$.

\subsection*{Acknowledgments}

The author thanks Eusebio Gardella for valuable comments and feedback.
I also want to thank the anonymous referee for his useful suggestions.

\subsection*{Notation}

Given a \ca{} $A$, we use $A_\sa$ to denote the set of self-adjoint elements in $A$. 
We write $M_d$ for the \ca{} of $d$-by-$d$ complex matrices.
Given a subset $F\subseteq A$ and $a\in A$, we write $a\in_\varepsilon F$ if there exists $b\in F$ such that $\|a-b\|<\varepsilon$.
We set $\NN:=\{0,1,2,\ldots\}$, the natural numbers including $0$.
The spectrum of an operator $a$ is denoted by $\sigma(a)$.

\section{The generator rank}
\label{sec:grPrelim}

In this section, we recall the definition and basic properties of the generator rank $\gr$ and its precursor $\grPre$ from \cite{Thi12arX:GenRnk}.

\begin{dfn}[{\cite[Definitions~2.1, 3.1]{Thi12arX:GenRnk}}]
\label{dfn:gr}
Let $A$ be a \ca{}.
Then $\grPre(A)$ is the defined as the smallest integer $n\geq 0$ such that for every $a_0,\ldots,a_n\in A_\sa$, $\varepsilon>0$ and $c\in A$, there exist $b_0,\ldots,b_n\in A_\sa$ such that
\[
\|b_j-a_j\|<\varepsilon \ \text{ for } j=0,\ldots,n, \andSep
c\in_\varepsilon C^*(b_0,\ldots,b_n).
\]
If no such $n$ exists, then $\grPre(A)=\infty$.
The \emph{generator rank} of $A$ is defined as $\gr(A):=\grPre(\widetilde{A})$, where $\widetilde{A}$ denotes the minimal unitization of $A$.
\end{dfn}

For the definition and the basic properties of the real rank of \ca{s}, we refer to \cite[Section~V.3.2, p.452ff]{Bla06OpAlgs}.

\begin{rmk}
Let $A$ be a \ca.
If $A$ is unital, then $\gr(A)=\grPre(A)$ by definition.
By \cite[Theorem~5.5]{Thi20arX:grSubhom}, we also have $\gr(A)=\grPre(A)$ whenever $A$ is subhomogeneous.
By \cite[Theorem~3.13]{Thi12arX:GenRnk}, we have $\gr(A)=\max\{\rr(A),\grPre(A)\}$.
In particular, if $A$ has real rank zero, then $\gr(A)=\grPre(A)$.

In general, however, it is unclear if $\gr(A)=\grPre(A)$;
see \cite[Question~3.16]{Thi12arX:GenRnk}.
\end{rmk}

For separable \ca{s}, the generator rank and its precursor can be described by the denseness of generating tuples.

\begin{thm}[{\cite[Theorem~3.4]{Thi12arX:GenRnk}}]
\label{prp:grCharSep}
Let $A$ be a \emph{separable} \ca{} and $n\in\NN$.
Then $\grPre(A)\leq n$ if and only if for every $a_0,\ldots,a_n\in A_\sa$ and $\varepsilon>0$, there exist $b_0,\ldots,b_n\in A_\sa$ such that
\[
\|b_j-a_j\|<\varepsilon \ \text{ for } j=0,\ldots,n, \andSep
A=C^*(b_0,\ldots,b_n).
\]

Analogously, we obtain a characterization of $\gr(A)\leq n$ by denseness of generating tuples in $\widetilde{A}$.
\end{thm}

We will repeatedly use the following permanence properties of the generator rank, which were shown in Theorems~6.2 and~6.3 of \cite{Thi12arX:GenRnk}.

\begin{thm}
\label{prp:permanence}
Let $I\subseteq A$ be a closed, two-sided ideal in a \ca{} $A$.
Then
\[
\max\big\{ \gr(I),\gr(A/I) \big\} 
\leq \gr(A) 
\leq \gr(I)+\gr(A/I)+1.
\]
Moreover, if $A=\varinjlim_\lambda A_\lambda$ is an inductive limit, then
\[
\gr(A) \leq \liminf_\lambda\gr(A_\lambda).
\]
\end{thm}

It is natural to expect that the generator rank of the direct sum of two \ca{s} is the maximum of the generator ranks of the summands.
We have verified this in the case that both summands have real rank zero (see \autoref{prp:gr_sum_rr0} below), and whenever one of the summands is subhomogeneous (\cite[Proposition~5.8]{Thi20arX:grSubhom}).
However, in general this remains unclear;
see \cite[Question~6.4]{Thi12arX:GenRnk}.

\begin{prp}[{\cite[Lemma~7.1]{Thi12arX:GenRnk}}]
\label{prp:gr_sum_rr0}
Let $A, B$ be \ca{s} of real rank zero.
Then $\gr(A\oplus B)=\max\{\gr(A),\gr(B)\}$.
\end{prp}

\begin{prp}[{\cite[Proposition~5.6]{Thi12arX:GenRnk}}]
\label{prp:gr0}
A \ca{} $A$ has generator rank zero if and only if $A$ is commutative with totally disconnected spectrum.
\end{prp}

\begin{rmk}
\label{rmk:gr1}
Let $A$ be a separable \ca.
Since two self-adjoint elements $a$ and $b$ generate the same sub-\ca{} as the element $a+ib$, and since the set of generators in $A$ is always a $G_\delta$-subset, we have $\grPre(A)\leq 1$ if and only if the set of (non-self-adjoint) generators in $A$ is a dense $G_\delta$-subset;
see \cite[Remark~3.7]{Thi12arX:GenRnk}.

If $A$ has real rank zero, then $\grPre(A)=\gr(A)$.
Thus, a separable \ca{} $A$ of real rank zero has generator rank at most one if and only if a generic element in $A$ is a generator.
\end{rmk}

We end this section with a standard result that will be used in the proof of \autoref{prp:construction}.
We include a proof for completeness.
For the definition and basic results of $C(X)$-algebras, we refer to \S{2} in \cite{Dad09CtsFieldsOverFD}.

Given an element $a$ in a $C(X)$-algebra, and $x\in X$, we let $a(x)$ denote the image of $a$ in the fiber at $x$.
We will repeatedly use that the map $X\to\RR$, $x\mapsto\|a(x)\|$, is upper semicontinuous, and that $\|a\|=\sup_{x\in X}\|a(x)\|$;
see \cite[Lemma~2.1]{Dad09CtsFieldsOverFD}.

\begin{lma}
\label{prp:genCenter}
Let $X$ be a compact, Hausdorff space, let $A$ be a unital $C(X)$-algebra, and let $B\subseteq A$ be a sub-\ca.
Then the following are equivalent:
\begin{enumerate}
\item[(a)]
We have $C(X)\subseteq B$;
\item[(b)]
$B$ separates the points in $X$ in the sense that for every distinct $x_0,x_1\in X$ there exists $b\in B$ with $b(x_0)=0$ and $b(x_1)=1$.
\end{enumerate}
\end{lma}
\begin{proof}
Assuming~(b), we need to verify~(a).
We first show that $1\in B$.
For every $y\in X$, there exists $b_y\in B$ with $b_y(y)=1$.
Replacing $b_y$ by $b_yb_y^*$, we may assume that $b_y$ is positive.
Set $U_y:=\{z\in X : \|1-b_y(z)\|<\tfrac{1}{2}\}$, which is open since $\|1-b_y(\cdot)\|$ is upper semicontinuous.
Hence, the compact set $X$ is covered by $(U_y)_{y\in X}$, which allows us to choose $y_1,\ldots,y_N\in X$ such that $X$ is covered by $\bigcup_{j=1}^N U_{y_j}$.
Set $b:=\sum_{j=1}^N b_{y_k}\in B$.
Let $f\colon\RR\to[0,1]$ be a continuous function with $f(0)=0$ and such that $f$ takes value $1$ on $[\tfrac{1}{2},\infty)$.
Set $c:=f(b)\in B$.
For each $x$, we have $b(x)\geq\tfrac{1}{2}$ and therefore $c(x)=1$.
It follows that $c=1$, since $\|c-1\|=\sup_{x\in X}\|c(x)-1\|=0$.

\emph{Claim~1: Let $F\subseteq X$ be closed and $x\in X\setminus F$.
Then there exists a positive $b\in B$ with $b|_F=0$ and $b(x)=1$.}
To prove the claim, let $y\in F$.
By assumption, there is $b_y\in B$ such that $b_y(y)=0$ and $b_y(x)=1$.
As above, we may assume that $b_y$ is positive, and we may then find $y_1,\ldots,y_N\in F$ such that $F$ is covered by the sets $U_j:=\{z\in X : \|b_{y_j}(z)\|<\tfrac{1}{2}\}$ for $j=1,\ldots,N$.
Let $g\colon\RR\to[0,1]$ be a continuous function with $g(1)=1$ and such that $g$ takes the value $0$ on $(-\infty,\tfrac{1}{2}]$.
For each $j$, set $c_j:=g(b_{y_j})\in B$.
Then $c_j$ vanishes on $U_{y_j}$, and $c_j(x)=1$.
Thus, $b:=(c_1c_2\cdots c_N)(c_1c_2\cdots c_N)^*$ has the desired properties, which proves the claim.

\emph{Claim~2: Let $F\subseteq G\subseteq X$ with $F$ closed and $G$ open.
Then there exists $b\in B$ with $0\leq b\leq 1$, $b|_{X\setminus G}=0$ and $b|_F=1$.}
To prove the claim, set $C:=X\setminus G$.
Applying Claim~1, and using that $1\in B$, for each $y\in C$ we obtain a positive $b_y\in B$ such that $b_y|_F=1$ and $b_y(y)=0$.
As in Claim~1, we find $y_1,\ldots,y_M$ such that $C$ is covered by the sets $\{z\in X : \|b_{y_j}(z)\|<\tfrac{1}{2}\}$ for $j=1,\ldots,N$.
With $g$ as above, the element $d:=g(b_{y_1})\cdots g(b_{y_M})$ satisfies $d|_C=0$ and $d|_F=1$.
Set $b:=h(dd^*)$, where $h\colon\RR\to\RR$ is the function $h(t)=\min\{t,1\}$.
Then $b$ has the claimed properties.

\emph{Claim~3: Let $f\colon X\to[0,\infty)$ be continuous, and let $\varepsilon>0$.
Then there exists $b\in B$ with $\|b-f\|\leq 2\varepsilon$.}
Set $b_0:=1$.
For each $k\geq 1$, set
\[
F_k := \big\{ x\in X : f(x)\geq k\varepsilon \big\}, \andSep 
G_k := \big\{ x\in X : f(x)>(k-1)\varepsilon \big\}.
\]
Then $G_k$ is open, $F_k$ is closed, and $F_k\subseteq G_k$.
Using Claim~2, choose $b_k\in B$ such that $0\leq b_k\leq 1$, $b_k|_{X\setminus G_k}=0$ and $b_k|_{F_k}=1$.

Note that $b_k=0$ for large enough $k$, which allows us to set $b:=\varepsilon\sum_{k=0}^\infty b_k$.
For each $x\in X$, we have $f(x)\leq b(x)\leq f(x)+2\varepsilon$ by construction, and thus $\|b(x)-f(x)\|\leq 2\varepsilon$.
Hence, $\|b-f\|=\sup_{x\in X}\|b(x)-f(x)\|\leq 2\varepsilon$.

It follows from Claim~3 that $B$ contains every positive function in $C([0,1])$, which implies the result.
\end{proof}

\section{Generator rank of C([0,1],A)}
\label{sec:tensI}

Throughout this section, we let $A$ denote a unital, separable \ca{} of real rank zero and generator rank at most one, and we set $B:=A\otimes C([0,1])$.
We consider $B$ with its natural $C([0,1])$-algebra structure, with each fiber isomorphic to $A$.
The goal is to verify $\gr(B)\leq 6$, which we accomplish in a series of Lemmas.

\begin{lma}
\label{prp:approxGenFibers}
Let $x_1,\ldots,x_4\in B_\sa$ and $\varepsilon>0$.
Then there exist $y_1,\ldots,y_4\in B_\sa$ such that
\[
\| y_k - x_k \| < \varepsilon\ \text{ for  } k=1,\ldots,4, \andSep A=C^*(y_1(t),\ldots,y_4(t))\ \text{ for all } t\in[0,1].
\]
\end{lma}
\begin{proof}
Using that each $x_j\colon[0,1]\to A$ is uniformly continuous, choose $N\geq 1$ such that $\|x_j(t)-x_j(t')\|<\tfrac{\varepsilon}{3}$ for all $t,t'\in[0,1]$ with $|t-t'|\leq\tfrac{1}{N}$.

Let $j\in\{0,\ldots,N\}$.
Then $x_1(\tfrac{2j}{2N})$ and $x_2(\tfrac{2j}{2N})$ belong to $A_\sa$.
Using that $A$ is unital and separable with $\gr(A)\leq 1$, we can apply \autoref{prp:grCharSep} to obtain $y_{1}^{(2j)},y_{2}^{(2j)}\in A_\sa$ such that
\[
\left\| y_{1}^{(2j)} - x_1(\tfrac{2j}{2N}) \right\| < \frac{\varepsilon}{3}, \quad
\left\| y_{2}^{(2j)} - x_2(\tfrac{2j}{2N}) \right\| < \frac{\varepsilon}{4}, \andSep
A = C^*(y_{1}^{(2j)},y_{2}^{(2j)}). 
\]
Analogously, we obtain $y_3^{(2j)},y_4^{(2j)}\in A_\sa$ such that
\[
\left\| y_3^{(2j)} - x_3(\tfrac{2j}{2N}) \right\| < \frac{\varepsilon}{3}, \quad
\left\| y_4^{(2j)} - x_4(\tfrac{2j}{2N}) \right\| < \frac{\varepsilon}{3}, \andSep
A = C^*(y_3^{(2j)},y_4^{(2j)}).
\]

For $j\in\{0,\ldots,N\}$ let $f_j,g_j\colon[0,1]\to[0,1]$ be continuous functions such that:
\begin{enumerate}
\item
$f_j$ takes the value $1$ on $[\tfrac{2j}{2N},\tfrac{2j+1}{2N}]$ for $j=0,\ldots,N-1$;
\item
the collection $(f_j)_{j=0,\ldots,N}$ is a partition of unity subordinate to the family $\{ [0,\tfrac{2}{2N}),(\tfrac{1}{2N},\tfrac{4}{2N}), \ldots,(\tfrac{2N-3}{2N},\tfrac{2N}{2N}), (\tfrac{2N-1}{2N},\tfrac{2N}{2N}] \}$;
\item
$g_j$ takes the value $1$ on $[\tfrac{2j-1}{2N},\tfrac{2j}{2N}]$ for $j=1,\ldots,N$;
\item
the collection $(g_j)_{j=0,\ldots,N}$ is a partition of unity subordinate to the family $\{ [0,\tfrac{1}{2N}),(\tfrac{0}{2N},\tfrac{3}{2N}), \ldots,(\tfrac{2N-4}{2N},\tfrac{2N-1}{2N}), (\tfrac{2N-2}{2N},\tfrac{2N}{2N}] \}$.
\end{enumerate}
The functions are depicted in the following picture:
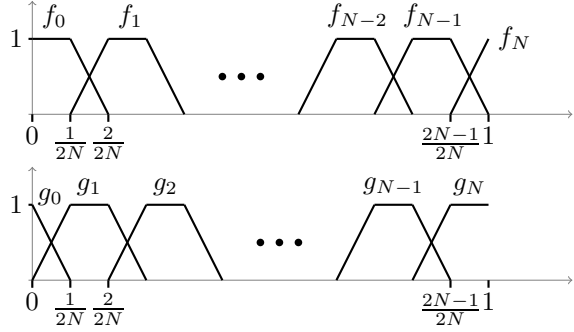
\[
\makebox{
\begin{tikzpicture}[xscale=0.5]
\draw [help lines, ->] (-0.1,0) -- (14.2,0);
\draw [help lines, ->] (0,-0.1) -- (0,1.5);
\draw [thick] (0,1) -- (1,1);
\draw [thick] (1,1) -- (2,0);
\draw [thick] (1,0) -- (2,1);
\draw [thick] (2,1) -- (3,1);
\draw [thick] (3,1) -- (4,0);
\draw [thick] (-0.1,1) -- (0,1);
\node [left] at (0,1) {$1$};
\node [right] at (0,1.3) {$f_0$};
\node [right] at (2,1.3) {$f_1$};

\draw [thick] (0,0) -- (0,-0.1);
\draw [thick] (1,0) -- (1,-0.1);
\draw [thick] (2,0) -- (2,-0.1);
\node [below] at (0,-0.05) {$0$};
\node [below] at (1,-0.05) {$\tfrac{1}{2N}$};
\node [below] at (2,-0.05) {$\tfrac{2}{2N}$};
\fill (5,0.5) ellipse (0.1 and 0.05);
\fill (5.5,0.5) ellipse (0.1 and 0.05);
\fill (6,0.5) ellipse (0.1 and 0.05);
\draw [thick] (7,0) -- (8,1);
\draw [thick] (8,1) -- (9,1);
\draw [thick] (9,1) -- (10,0);
\draw [thick] (9,0) -- (10,1);
\draw [thick] (10,1) -- (11,1);
\draw [thick] (11,1) -- (12,0);
\draw [thick] (11,0) -- (12,1);
\node [right] at (7.5,1.3) {$f_{N-2}$};
\node [right] at (9.5,1.3) {$f_{N-1}$};
\node [right] at (12,1.0) {$f_{N}$};
\draw [thick] (11,0) -- (11,-0.1);
\draw [thick] (12,0) -- (12,-0.1);
\node [below] at (11,-0.05) {$\tfrac{2N-1}{2N}$};
\node [below] at (12,-0.05) {$1$};

\draw [help lines, ->] (-0.1,-2.2) -- (14.2,-2.2);
\draw [help lines, ->] (0,-2.3) -- (0,-0.7);
\draw [thick] (0,-1.2) -- (1,-2.2);
\draw [thick] (0,-2.2) -- (1,-1.2);
\draw [thick] (1,-1.2) -- (2,-1.2);
\draw [thick] (2,-1.2) -- (3,-2.2);
\draw [thick] (2,-2.2) -- (3,-1.2);
\draw [thick] (3,-1.2) -- (4,-1.2);
\draw [thick] (4,-1.2) -- (5,-2.2);
\draw [thick] (-0.1,-1.2) -- (0,-1.2);
\node [left] at (0,-1.2) {$1$};
\node [right] at (-0.1,-1.1) {$g_0$};
\node [above] at (1.5,-1.2) {$g_1$};
\node [above] at (3.5,-1.2) {$g_2$};
\draw [thick] (0,-2.2) -- (0,-2.3);
\draw [thick] (1,-2.2) -- (1,-2.3);
\draw [thick] (2,-2.2) -- (2,-2.3);
\node [below] at (0,-2.25) {$0$};
\node [below] at (1,-2.25) {$\tfrac{1}{2N}$};
\node [below] at (2,-2.25) {$\tfrac{2}{2N}$};
\fill (6,-1.7) ellipse (0.1 and 0.05);
\fill (6.5,-1.7) ellipse (0.1 and 0.05);
\fill (7,-1.7) ellipse (0.1 and 0.05);
\draw [thick] (8,-2.2) -- (9,-1.2);
\draw [thick] (9,-1.2) -- (10,-1.2);
\draw [thick] (10,-1.2) -- (11,-2.2);
\draw [thick] (10,-2.2) -- (11,-1.2);
\draw [thick] (11,-1.2) -- (12,-1.2);
\node [above] at (9.5,-1.2) {$g_{N-1}$};
\node [above] at (11.5,-1.2) {$g_{N}$};
\draw [thick] (11,-2.2) -- (11,-2.3);
\draw [thick] (12,-2.2) -- (12,-2.3);
\node [below] at (11,-2.25) {$\tfrac{2N-1}{2N}$};
\node [below] at (12,-2.25) {$1$};
\end{tikzpicture}
}
\]

Set
\[
y_1 := \sum_{j=0}^{N} y_1^{(2j)} f_{j}, \quad
y_2 := \sum_{j=0}^{N} y_2^{(2j)} f_{j}, \quad
y_3 := \sum_{j=0}^{N} y_3^{(2j)} g_{j}, \quad
y_4 := \sum_{j=0}^{N} y_4^{(2j)} g_{j}.
\]

To verify that $\|y_1-x_1\|<\varepsilon$, we estimate $\|y_1(t)-x_1(t)\|$ for $t\in[0,1]$.
For $t=1$, we have $y_1(1)=y_1^{(2N)}$, and so
\[
\left\| y_1(1) - x_1(1) \right\| 
= \left\| y_1(1) - y_1^{(2N)} \right\| < \frac{\varepsilon}{3} \leq \frac{2}{3}\varepsilon.
\]
Given $t\in[0,1)$, let $j\in\{1,\ldots,N\}$ satisfy $t\in[\tfrac{2j-2}{2N},\tfrac{2j}{2N})$.
Since $|t-\tfrac{2j-2}{2N}|,|t-\tfrac{2j}{2N}|<\tfrac{2}{2N}$, we have
\begin{align*}
\left\| y_1^{(2j-2)} - x_1(t) \right\|
&\leq \left\| y_1^{(2j-2)} - x_1(\tfrac{2j-2}{2N}) \right\| 
+ \left\| x_1(\tfrac{2j-2}{2N}) - x_1(t) \right\|
< \frac{\varepsilon}{3} + \frac{\varepsilon}{3} 
\leq \frac{2}{3}\varepsilon, \\
\left\| y_1^{(2j)} - x_1(t) \right\|
&\leq \left\| y_1^{(2j)} - x_1(\tfrac{2j}{2N}) \right\| 
+ \left\|  x_1(\tfrac{2j}{2N}) - x_1(t) \right\|
< \frac{\varepsilon}{3} + \frac{\varepsilon}{3} 
\leq \frac{2}{3}\varepsilon.
\end{align*}
Using that $y_1(t)= y_1^{(2j-2)}f_{j-1}(t)+y_1^{(2j)}f_j(t)$, and that $f_{j-1}(t)+f_{j}(t)=1$, we get
\[
\|y_1(t)-x_1(t)\|
\leq \left\| y_1^{(2j-2)} - x_1(t) \right\|f_{j-1}(t)
+ \left\| y_1^{(2j)} - x_1(t) \right\|f_{j}(t)
\leq\frac{2}{3}\varepsilon.
\]

Hence, $\|y_1-x_1\|=\sup_{t\in[0,1]}\|y_1(t)-x_1(t)\|\leq\tfrac{2}{3}\varepsilon<\varepsilon$.
Analogously, one shows that $\|y_k-x_k\|<\varepsilon$ for $k=2,3,4$.

It remains to verify that $\{y_1,\ldots,y_4\}$ generates $A$ in each fiber. 
Given $t\in[0,1]$, choose $l\in\{0,\ldots,2N-1\}$ such that $t\in[\tfrac{l}{2N},\tfrac{l+1}{2N}]$.
If $l$ is even, set $j=\tfrac{l}{2}$.
Then
\[
y_1(t) = y_1^{(2j)}, \andSep
y_2(t) = y_2^{(2j)}.
\]
If $l$ is odd, set $j=\tfrac{l+1}{2}$.
Then
\[
y_3(t) = y_3^{(2j)}, \andSep
y_4(t) = y_4^{(2j)}.
\]

Thus, either $\{y_1(t),y_2(t)\}$ or $\{y_3(t),y_4(t)\}$ generates $A$, and it follows that $A=C^*(y_1(t),\ldots,y_4(t))$ in either case.
\end{proof}

\begin{lma}
\label{prp:approxSepInterval}
Let $x_1,x_2,x_3\in B_\sa$ and $\varepsilon>0$.
Then there exist $y_1,y_2,y_3\in B_\sa$ with $\|y_k-x_k\|<\varepsilon$ for $k=1,2,3$, and such that $C^*(y_1,y_2,y_3)$ separates the points in $[0,1]$ in the sense of \autoref{prp:genCenter}(b).
\end{lma}
\begin{proof}
Using that each $x_j\colon[0,1]\to A$ is uniformly continuous, choose $N\geq 1$ such that $\|x_j(t)-x_j(t')\|<\tfrac{\varepsilon}{4}$ for all $t,t'\in[0,1]$ with $|t-t'|\leq\tfrac{1}{N}$.

Let $j\in\{0,\ldots,N\}$.
Using that $A$ has real rank zero, we find invertible elements $y_1^{(3j)},y_2^{(3j)},y_3^{(3j)}\in A_\sa$ with finite spectra such that
\[
\left\| y_1^{(3j)} - x_1(\tfrac{3j}{3N}) \right\| < \frac{\varepsilon}{4}, \quad
\left\| y_1^{(3j)} - x_2(\tfrac{3j}{3N}) \right\| < \frac{\varepsilon}{4}, \andSep
\left\| y_1^{(3j)} - x_3(\tfrac{3j}{3N}) \right\| < \frac{\varepsilon}{4}.
\]

By perturbing the elements, if necessary, we may assume that the spectra of $y_k^{(3j)}$ and $y_{k'}^{(3j')}$ are disjoint whenever $k\neq k'$ or $j\neq j'$.
Choose $\mu>0$ such that any two distinct points in $\{0\}\cup\bigcup_{k,j}\sigma(y_k^{(3j)})$ have distance at least $2\mu$.
We may assume that $\mu<\tfrac{\varepsilon}{4}$.

For $j\in\{0,\ldots,N\}$ let $f_j,g_j,h_j\colon[0,1]\to[0,1]$ be continuous functions such that:
\begin{enumerate}
\item
$f_j$ takes the value $1$ on $[\tfrac{3j}{3N},\tfrac{3j+2}{3N}]$ for $j=0,\ldots,N-1$;
\item
the collection $(f_j)_{j=0,\ldots,N}$ is a partition of unity subordinate to the family $\{ [0,\tfrac{3}{3N}),(\tfrac{2}{3N},\tfrac{6}{2N}), \ldots,(\tfrac{3N-4}{3N},\tfrac{3N}{3N}), (\tfrac{3N-1}{3N},\tfrac{3N}{3N}] \}$;
\item
$g_j$ takes the value $1$ on $[\tfrac{3j-2}{3N},\tfrac{3j}{2N}]$ for $j=1,\ldots,N$;
\item
the collection $(g_j)_{j=0,\ldots,N}$ is a partition of unity subordinate to the family $\{ [0,\tfrac{1}{3N}),(\tfrac{0}{3N},\tfrac{4}{3N}), \ldots,(\tfrac{3N-6}{3N},\tfrac{3N-2}{3N}), (\tfrac{3N-3}{3N},\tfrac{3N}{3N}] \}$.
\item
$h_0$ takes the value $1$ on $[\tfrac{0}{3N},\tfrac{1}{3N}]$;
$h_N$ takes the value $1$ on $[\tfrac{3N-1}{3N},\tfrac{3N}{3N}]$;
and $h_j$ takes the value $1$ on $[\tfrac{3j-1}{3N},\tfrac{3j+1}{3N}]$ for $j=1,\ldots,N-1$;
\item
the collection $(h_j)_{j=0,\ldots,N}$ is a partition of unity subordinate to the family $\{ [0,\tfrac{2}{3N}),(\tfrac{1}{3N},\tfrac{5}{2N}), \ldots,(\tfrac{3N-5}{3N},\tfrac{3N-1}{3N}), (\tfrac{3N-2}{3N},\tfrac{3N}{3N}] \}$.
\end{enumerate} 
The functions are depicted in the following picture:
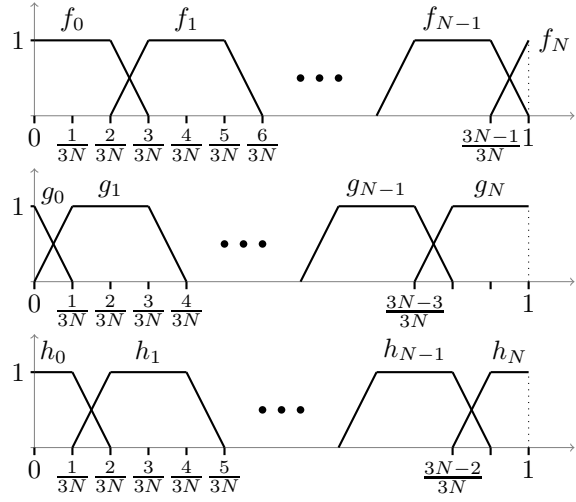
\[
\makebox{
\begin{tikzpicture}[xscale=0.5]
\draw [help lines, ->] (-0.1,0) -- (14.2,0);
\draw [help lines, ->] (0,-0.1) -- (0,1.5);

\draw [thick] (0,1) -- (2,1);
\draw [thick] (2,1) -- (3,0);
\draw [thick] (2,0) -- (3,1);
\draw [thick] (3,1) -- (5,1);
\draw [thick] (5,1) -- (6,0);
\fill (7,0.5) ellipse (0.1 and 0.05);
\fill (7.5,0.5) ellipse (0.1 and 0.05);
\fill (8,0.5) ellipse (0.1 and 0.05);
\draw [thick] (9,0) -- (10,1);
\draw [thick] (10,1) -- (12,1);
\draw [thick] (12,1) -- (13,0);
\draw [thick] (12,0) -- (13,1);
\draw [dotted] (13,0) -- (13,1);

\draw [thick] (-0.1,1) -- (0,1);
\node [left] at (0,1) {$1$};
\node [above] at (1,1) {$f_0$};
\node [above] at (4,1) {$f_1$};
\node [above] at (11,1) {$f_{N-1}$};
\node [right] at (13,1.0) {$f_{N}$};

\draw [thick] (0,0) -- (0,-0.1);
\draw [thick] (1,0) -- (1,-0.1);
\draw [thick] (2,0) -- (2,-0.1);
\draw [thick] (3,0) -- (3,-0.1);
\draw [thick] (4,0) -- (4,-0.1);
\draw [thick] (5,0) -- (5,-0.1);
\draw [thick] (6,0) -- (6,-0.1);
\draw [thick] (12,0) -- (12,-0.1);
\draw [thick] (13,0) -- (13,-0.1);
\node [below] at (0,-0.05) {$0$};
\node [below] at (1,-0.05) {$\tfrac{1}{3N}$};
\node [below] at (2,-0.05) {$\tfrac{2}{3N}$};
\node [below] at (3,-0.05) {$\tfrac{3}{3N}$};
\node [below] at (4,-0.05) {$\tfrac{4}{3N}$};
\node [below] at (5,-0.05) {$\tfrac{5}{3N}$};
\node [below] at (6,-0.05) {$\tfrac{6}{3N}$};
\node [below] at (12,-0.05) {$\tfrac{3N-1}{3N}$};
\node [below] at (13,-0.05) {$1$};


\draw [help lines, ->] (-0.1,-2.2) -- (14.2,-2.2);
\draw [help lines, ->] (0,-2.3) -- (0,-0.7);

\draw [thick] (0,-1.2) -- (1,-2.2);
\draw [thick] (0,-2.2) -- (1,-1.2);
\draw [thick] (1,-1.2) -- (3,-1.2);
\draw [thick] (3,-1.2) -- (4,-2.2);
\fill (5,-1.7) ellipse (0.1 and 0.05);
\fill (5.5,-1.7) ellipse (0.1 and 0.05);
\fill (6,-1.7) ellipse (0.1 and 0.05);
\draw [thick] (7,-2.2) -- (8,-1.2);
\draw [thick] (8,-1.2) -- (10,-1.2);
\draw [thick] (10,-1.2) -- (11,-2.2);
\draw [thick] (10,-2.2) -- (11,-1.2);
\draw [thick] (11,-1.2) -- (13,-1.2);
\draw [dotted] (13,-2.2) -- (13,-1.2);

\draw [thick] (-0.1,-1.2) -- (0,-1.2);
\node [left] at (0,-1.2) {$1$};
\node [above] at (0.5,-1.3) {$g_0$};
\node [above] at (2,-1.2) {$g_1$};
\node [above] at (9,-1.2) {$g_{N-1}$};
\node [above] at (12,-1.2) {$g_{N}$};

\draw [thick] (0,-2.2) -- (0,-2.3);
\draw [thick] (1,-2.2) -- (1,-2.3);
\draw [thick] (2,-2.2) -- (2,-2.3);
\draw [thick] (3,-2.2) -- (3,-2.3);
\draw [thick] (4,-2.2) -- (4,-2.3);
\draw [thick] (10,-2.2) -- (10,-2.3);
\draw [thick] (11,-2.2) -- (11,-2.3);
\draw [thick] (12,-2.2) -- (12,-2.3);
\draw [thick] (13,-2.2) -- (13,-2.3);
\node [below] at (0,-2.25) {$0$};
\node [below] at (1,-2.25) {$\tfrac{1}{3N}$};
\node [below] at (2,-2.25) {$\tfrac{2}{3N}$};
\node [below] at (3,-2.25) {$\tfrac{3}{3N}$};
\node [below] at (4,-2.25) {$\tfrac{4}{3N}$};
\node [below] at (10,-2.25) {$\tfrac{3N-3}{3N}$};
\node [below] at (13,-2.25) {$1$};


\draw [help lines, ->] (-0.1,-4.4) -- (14.2,-4.4);
\draw [help lines, ->] (0,-4.5) -- (0,-2.9);

\draw [thick] (0,-3.4) -- (1,-3.4);
\draw [thick] (1,-3.4) -- (2,-4.4);
\draw [thick] (1,-4.4) -- (2,-3.4);
\draw [thick] (2,-3.4) -- (4,-3.4);
\draw [thick] (4,-3.4) -- (5,-4.4);
\fill (6,-3.9) ellipse (0.1 and 0.05);
\fill (6.5,-3.9) ellipse (0.1 and 0.05);
\fill (7,-3.9) ellipse (0.1 and 0.05);
\draw [thick] (8,-4.4) -- (9,-3.4);
\draw [thick] (9,-3.4) -- (11,-3.4);
\draw [thick] (11,-3.4) -- (12,-4.4);
\draw [thick] (11,-4.4) -- (12,-3.4);
\draw [thick] (12,-3.4) -- (13,-3.4);
\draw [dotted] (13,-4.4) -- (13,-3.4);

\draw [thick] (-0.1,-3.4) -- (0,-3.4);
\node [left] at (0,-3.4) {$1$};
\node [above] at (0.5,-3.4) {$h_0$};
\node [above] at (3,-3.4) {$h_1$};
\node [above] at (10,-3.4) {$h_{N-1}$};
\node [above] at (12.5,-3.4) {$h_{N}$};

\draw [thick] (0,-4.4) -- (0,-4.5);
\draw [thick] (1,-4.4) -- (1,-4.5);
\draw [thick] (2,-4.4) -- (2,-4.5);
\draw [thick] (3,-4.4) -- (3,-4.5);
\draw [thick] (4,-4.4) -- (4,-4.5);
\draw [thick] (5,-4.4) -- (5,-4.5);
\draw [thick] (11,-4.4) -- (11,-4.5);
\draw [thick] (12,-4.4) -- (12,-4.5);
\draw [thick] (13,-4.4) -- (13,-4.5);
\node [below] at (0,-4.45) {$0$};
\node [below] at (1,-4.45) {$\tfrac{1}{3N}$};
\node [below] at (2,-4.45) {$\tfrac{2}{3N}$};
\node [below] at (3,-4.45) {$\tfrac{3}{3N}$};
\node [below] at (4,-4.45) {$\tfrac{4}{3N}$};
\node [below] at (5,-4.45) {$\tfrac{5}{3N}$};
\node [below] at (11,-4.45) {$\tfrac{3N-2}{3N}$};
\node [below] at (13,-4.45) {$1$};
\end{tikzpicture}
}
\]

Define $y_1,y_2,y_3\colon[0,1]\to A$ by 
\begin{align*}
y_1(t) &:= \sum_{j=0}^{N} \left( y_1^{(3j)} + t\mu \right) f_{j}(t), \quad
y_2(t) := \sum_{j=0}^{N} \left( y_2^{(3j)} + t\mu \right) g_{j}(t), \\
y_3(t) &:= \sum_{j=0}^{N} \left( y_3^{(3j)} + t\mu \right) h_{j}(t),
\end{align*}
for $t\in[0,1]$.

To verify that $\|y_1-x_1\|<\varepsilon$, we estimate $\|y_1(t)-x_1(t)\|$ for $t\in[0,1]$, similarly as in the proof of \autoref{prp:approxGenFibers}.
For $t=1$, we obtain
\[
\left\| y_1(1) - x_1(1) \right\| 
= \left\|  y_1^{(3j)} + \mu - x_1(1) \right\|
< \frac{\varepsilon}{4} + \frac{\varepsilon}{4} 
\leq \frac{3}{4}\varepsilon.
\]
Given $t\in[0,1)$, let $j\in\{1,\ldots,N\}$ satisfy $t\in[\tfrac{3j-3}{3N},\tfrac{3j}{3N})$.
Using at the second step that $|t-\tfrac{3j}{3N}|<\tfrac{3}{3N}$, we have
\begin{align*}
\left\| y_1^{(3j)} + t\mu - x_1(t) \right\|
&\leq \mu + \left\| y_1^{(3j)} - x_1(\tfrac{3j}{3N}) \right\| 
+ \left\| x_1(\tfrac{3j}{3N}) - x_1(t) \right\|
\leq \frac{3}{4}\varepsilon.
\end{align*}
Analogously, we obtain $\| y_1^{(3j-3)} + t\mu - x_1(t) \|\leq \frac{3}{4}\varepsilon$.
We deduce that
\[
\left\| y_1(t) - x_1(t) \right\|
= \left\| \left( \left(y_1^{(3j-3)} + t\mu\right)f_{j-1}(t) + \left( y_1^{(3j)} + t\mu\right) f_j(t) \right) - x_1(t) \right\|
\leq \frac{3}{4}\varepsilon.
\]
Hence, $\|y_1-x_1\|=\sup_{t\in[0,1]}\|y_1(t)-x_1(t)\|\leq\tfrac{3}{4}\varepsilon<\varepsilon$.
Analogously, one shows that $\|y_k-x_k\|<\varepsilon$ for $k=2,3$.

It remains to verify that $C^*(y_1,y_2,y_3)$ separates the points in $[0,1]$ in the sense of \autoref{prp:genCenter}(b).
Let $s,t\in[0,1]$ be distinct.
Choose $k\in\{1,2,3\}$ such that both $s$ and $t$ are contained in the set
\[
\bigcup_{j\in\ZZ} \left[ \frac{3j+k-1}{3N},\frac{3j+k+1}{3N} \right]
\]
Let us consider the case $k=1$.
Choose $j,j'\in\{0,\ldots,N\}$ such that $s\in[\tfrac{3j}{3N},\frac{3j+2}{3N}]$ and $t\in[\tfrac{3j'}{3N},\frac{3j'+2}{3N}]$.
Then
\[
y_1(s) = y_1^{(3j)} + s\mu , \andSep
y_1(t) = y_1^{(3j')} + t\mu.
\]
Since $s\neq t$, and by choice of $\mu$, it follows that $y_1(s)$ and $y_1(t)$ have finite spectra that are disjoint and do not contain $0$.
Hence, there exists a continuous function $f\colon\RR\to[0,1]$ such that $f(y_1(s))=0$ and $f(y_1(t))=1$.
Then the element $b:=f(y_1)\in B$ satisfies $b(s)=0$ and $b(t)=1$.

The cases $k=2,3$ are analogous, using $y_2$ and $y_3$.
\end{proof}

\begin{prp}
\label{prp:tensI}
Let $A$ be a unital, separable \ca{} of real rank zero and generator rank at most one.
Set $B:=A\otimes C([0,1])$.
The we have $\gr(B)\leq 6$.
\end{prp}
\begin{proof}
We show that every $7$-tuple in $B_\sa$ can be approximated by a tuple that generates $B$.
Since $B$ is separable and unital, this verifies $\gr(B)\leq 6$;
see \autoref{prp:grCharSep}.

Let $x_1,\ldots,x_7\in B_\sa$ and $\varepsilon>0$.
Applying \autoref{prp:approxGenFibers} for $x_1,\ldots,x_4$, we obtain $y_1,\ldots,y_4\in B_\sa$ such that 
\[
\| y_j - x_j \|<\varepsilon \ \text{ for } j=1,\ldots,4, \andSep
A=C^*(y_1(t),\ldots,y_4(t)) \ \text{ for all } t\in[0,1].
\]
Applying \autoref{prp:approxSepInterval} for $x_5,x_6,x_7$, we obtain $y_5,y_6,y_7\in B_\sa$ such that $\|y_k-x_k\|<\varepsilon$ for $k=5,6,7$, and such that $C^*(y_5,y_6,y_7)$ separates the points in $[0,1]$ in the sense of \autoref{prp:genCenter}(b).

Set $D:=C^*(y_1,\ldots,y_7)\subseteq B$.
Then $D$ exhausts each fiber of $B$, and moreover separates the points in $[0,1]$.
Hence, $D=B$ by \cite[Lemma~3.2]{ThiWin14GenZStableCa}, which shows that $\{y_1,\ldots,y_7\}$ generates $B$, as desired.
\end{proof}

\begin{rmk}
The proof of \autoref{prp:tensI} can be generalized to show the following:
If $A$ is a unital, separable \ca{} of real rank zero, then $B:=A\otimes C([0,1])$ satisfies $\gr(B)\leq 2\gr(A)+4$.
\end{rmk}

\section{Establishing finite generator rank}
\label{sec:finiteGr}

In this section we prove that unital, separable, $\mathcal{Z}$-stable \ca{s} of real rank zero have finite generator rank;
see \autoref{prp:gr8}.
In the next section, we will successively reduce the upper bound for the generator rank of such algebras down to one.

We start with a lemma that simplifies the computation of the generator rank for \ca{s} that absorb a strongly self-absorbing \ca.
For the definition and basic results of strongly self-absorbing \ca{s} we refer to \cite{TomWin07ssa}.
Given a strongly self-absorbing \ca{} $D$, a \ca{} $A$ is said to be \emph{$D$-stable} if $A\cong A\otimes D$.
Since every strongly self-absorbing \ca{} is nuclear, we do not need to specify the tensor product.
Typical examples of strongly self-absorbing \ca{s} are UHF-algebras of infinite type, the Jiang-Su algebra $\mathcal{Z}$, and the Cuntz algebras $\mathcal{O}_\infty$ and $\mathcal{O}_2$.

\begin{lma}
\label{prp:genDabsorbing}
Let $D$ be a strongly self-absorbing \ca{}, let $A$ be a separable, $D$-stable \ca, and $n\in\NN$.
Then the following are equivalent:
\begin{enumerate}
\item
We have $\grPre(A)\leq n$;
\item
For every $x_0,\ldots,x_n\in A_\sa$ and $\varepsilon>0$, there exist $y_0,\ldots,y_n\in (A\otimes D)_\sa$ such that
\[
\quad \quad 
\| y_j - (x_j\otimes 1) \|<\varepsilon \ \text{ for } j=0,\ldots,n, \andSep
A\otimes D = C^*(y_0,\ldots,y_n);
\]
\item
For every $x_0,\ldots,x_n\in A_\sa$, $\varepsilon>0$, and $z\in A$, there exist $y_0,\ldots,y_n\in (A\otimes D)_\sa$ such that
\[
\quad \quad 
\| y_j - (x_j\otimes 1) \|<\varepsilon\ \text{ for } j=0,\ldots,n, \andSep 
z\otimes 1\in_\varepsilon C^*(y_0,\ldots,y_k).
\]
\end{enumerate}
\end{lma}
\begin{proof}
Since $A\cong A\otimes D$, we have $\grPre(A)=\grPre(A\otimes D)$.
Using that $A$ is separable, it follows from \autoref{prp:grCharSep} that~(1) implies~(2).
It is clear that~(2) implies~(3).
Assuming~(3), let us verify~(1).
To show $\grPre(A\otimes D)\leq n$, let $a_0,\ldots,a_n\in (A\otimes D)_\sa$, $\varepsilon>0$ and $c\in A\otimes D$.
We need to find $b_0,\ldots,b_n\in (A\otimes D)_\sa$ such that
\[
\| b_j - a_j \| < \varepsilon \text{ for } j=0,\ldots,n, \andSep
c\in_\varepsilon C^*(b_0,\ldots,b_n).
\]

Since $D$ is strongly self-absorbing and $A$ is separable and $D$-stable, there exists a ${}^*$-isomorphism $\Phi\colon A\to A\otimes D$ that is approximately unitarily equivalent to the inclusion $\iota\colon A\to A\otimes D$ given by $\iota(a)=a\otimes 1$, that is, there exists a sequence $(u_m)_m$ of unitaries in $A\otimes D$ such that $\lim_{m\to\infty} u_m\Phi(a)u_m^* = \iota(a)$ for all $a\in A$;
see \cite[Theorem~2.2]{TomWin07ssa}.
Set
\[
x_j := \Phi^{-1}(a_j)\ \text{ for } j=0,\ldots,n, \andSep
z:=\Phi^{-1}(c).
\]
Using that $u_m^*(x_j\otimes 1)u_m\to a_j$ for $j=0,\ldots,n$, and $u_m^*(z\otimes 1)u_m\to c$, we can choose $m$ such that
\[
\left\| u_m^*(x_j\otimes 1)u_m-a_j \right\| < \frac{\varepsilon}{2}\ \text{ for } j=0,\ldots,n, \andSep
\left\| u_m^*(z\otimes 1)u_m-c \right\| < \frac{\varepsilon}{2}.
\]
By assumption~(3), we obtain $y_0,\ldots,y_n\in (A\otimes D)_\sa$ such that
\[
\| y_j - (x_j\otimes 1) \| < \frac{\varepsilon}{2}\ \text{ for } j=0,\ldots,n, \andSep
z\otimes 1 \in_{\varepsilon/2} C^*(y_0,\ldots,y_{n}).
\]
For $j=0,\ldots,n$, we have
\[
\|u_m^* y_j u_m - a_j\|
\leq \|u_m^* y_j u_m - u_m^* (x_j\otimes 1) u_m\| + \| u_m^* (x_j\otimes 1) u_m - a_j \| < \varepsilon.
\]
Moreover, from $z\otimes 1 \in_{\varepsilon/2} C^*(y_0,\ldots,y_n)$ it follows that
\[
u_m^*(z\otimes 1)u_m \in_{\varepsilon/2} C^*(u_m^* y_0 u_m,\ldots, u_m^* y_n u_m),
\]
and thus
\[
c\in_\varepsilon C^*(u_m^* y_0 u_m,\ldots, u_m^* y_n u_m),
\]
which shows that $u_m^*y_0u_m,\ldots,u_m^*y_nu_m$ have the desired properties.
\end{proof}

\begin{prp}
\label{prp:UHFabs}
Let $A$ be a separable, unital \ca{} of real rank zero that tensorially absorbs a UHF-algebra of infinite type.
Then $\gr(A)\leq 1$.
\end{prp}
\begin{proof}
Let $D$ be UHF-algebra of infinite type such that $A$ is $D$-stable.
Since $A$ is unital, we have $\gr(A)=\grPre(A)$.
To verify condition~(3) of \autoref{prp:genDabsorbing}, let $a,b\in A_\sa$, $\varepsilon>0$ and $c\in A$.
We need to find $x,y\in (A\otimes D)_\sa$ such that
\[
\| x - (a\otimes 1) \| < \varepsilon, \quad
\| y - (b\otimes 1) \| < \varepsilon, \andSep
c\otimes 1 \in_\varepsilon C^*(x,y).
\]

Since $A$ has real rank zero, we may assume that $a$ is invertible and that its spectrum $\sigma(a)$ is finite, so that there exist $\lambda_1,\ldots,\lambda_n\in\RR\setminus\{0\}$ and pairwise orthogonal projections $p_1,\ldots,p_n\in A$ that sum to $1_A$ such that $a=\sum_{j=1}^n \lambda_j p_j$.
Choose $\mu>0$ such that $\mu$ is strictly smaller than the distance between any two values in $\sigma(a)\cup\{0\}$.
We may assume that $\mu<\varepsilon$.

Choose $d\geq 5$ such that $D$ admits a unital embedding $M_d\subseteq D$.
Using that $c$ can be written as a linear combination of four positive, invertible elements, we can choose positive, invertible elements $c_2,c_3\ldots,c_d\in A$ such that
\[
\|c_j\| < \frac{\varepsilon}{d}, \ \text{ for } j=2,\ldots,d, \andSep
c \in C^*(c_2,\ldots,c_d).
\]
Let $(e_{j,k})_{j,k=1,\ldots,d}$ be matrix units for $M_d$.
We define $x,y\in A\otimes M_{d}$ as
\begin{align*}
x := a\otimes 1 + \sum_{j=1}^d \frac{j}{d}\mu e_{jj}, \andSep
y := b\otimes 1 + \sum_{j=2}^d c_j\otimes \left( e_{1j} + e_{j1} \right) .
\end{align*}

As matrices, these elements look as follows:
\[
x=\left(
\begin{array}{cccccc}
a + \frac{1}{d}\mu & 0 & \cdots & 0 \\
0 & a + \frac{2}{d}\mu & & \vdots \\
\vdots & & \ddots & 0 \\
0 & \cdots & 0 & a + \mu \\
\end{array}\right), \
y=\left(
\begin{array}{cccccc}
b & c_2 & c_3 & \cdots & c_d \\
c_2 & b & 0 & \cdots & 0 \\
c_3 & 0 & b & & \vdots \\
\vdots & \vdots & & \ddots & 0 \\
c_d & 0 & \cdots & 0 & b \\
\end{array}\right), \
\]

Then 
\[
\| x - (a\otimes 1) \|
= \left\| \sum_{j=1}^d \frac{j}{d}\mu e_{jj} \right\| 
= \mu < \varepsilon,
\]
and
\begin{align*}
\| y - (b\otimes 1) \|
= \left\| \sum_{j=2}^d c_j\otimes \left( e_{1j} + e_{j1} \right) \right\| 
\leq \sum_{j=2}^d \|c_j\| < \varepsilon.
\end{align*}

Set $B:=C^*(x,y)\subseteq A\otimes M_{d}\subseteq A\otimes D$.
We will verify that $z\otimes 1 \in B$.
The $j$-th element on the diagonal of $x$ is $a+\tfrac{j}{d}\mu$, whose spectrum is
\[
\sigma(a+\frac{j}{d}\mu) 
= \left\{ \lambda_1+\frac{j}{d}\mu,\ldots,\lambda_n+\frac{j}{d}\mu \right\}.
\]

Given distinct $j,k\in\{1,\ldots,d\}$, it follows from the choice of $\mu$ that the spectra of $a+\tfrac{j}{d}\mu$ and $a+\tfrac{k}{d}\mu$ are finite disjoint sets not containing $0$.
For $j\in\{1,\ldots,d\}$, let $f_j\colon\RR\to[0,1]$ be a continuous functions that takes the value $1$ on $\sigma(a+\tfrac{j}{d}\mu)$, and that takes the value $0$ on $\{0\}\cup\bigcup_{k\neq j}\sigma(a+\frac{k}{d}\mu)$.
Then
\[
1\otimes e_{jj} = f_j(x) \in B \subseteq A\otimes M_d.
\]
Thus, $B$ contains the diagonal matrix units of $M_d$.

To show that $B$ also contains the other matrix units, we follow ideas of Olsen and Zame from \cite{OlsZam76CaSingleGen}.
Given $k\in\{2,\ldots,d\}$, we have 
\[
c_k\otimes e_{1k} = (1\otimes e_{11}) y (1\otimes e_{kk}) \in B.
\]
Then
\[
c_k^2 \otimes e_{11} = (c_j\otimes e_{1k})(c_j\otimes e_{1k})^* \in B.
\]

Since $c_k$ is positive and invertible, we have $c_k^{-1}\in C^*(c_k^2)\subseteq A$, and it follows that $c_k^{-1}\otimes e_{11}\in B$.
Hence
\begin{align*}
1\otimes e_{1k} =
(c_k^{-1}\otimes e_{11})(c_k\otimes e_{1k}) \in B.
\end{align*}

It follows that $1\otimes M_d\subseteq B$.
For each $k\in\{2,\ldots,d\}$, we deduce that
\[
c_k\otimes 1
= \sum_{j=1}^d (1\otimes e_{j1})(c_k\otimes e_{1k})(1\otimes e_{kj}) \in B.
\]
Since $c\in C^*(c_2,\ldots,c_d)$, we get $c\otimes 1\in B\subseteq A\otimes D$, as desired.
\end{proof}

We use $M_{2^\infty}$ to denote the UHF-algebra of type $2^\infty$, and similarly for $M_{3^\infty}$.

\begin{prp}
\label{prp:gr8}
Let $A$ be a unital, separable \ca{} of real rank zero.
Then $\gr(A\otimes\mathcal{Z})\leq 8$.
\end{prp}
\begin{proof}
Let $Z_{2^\infty,3^\infty}$ be the generalized dimension-drop algebra given by
\[
Z_{2^\infty,3^\infty} = \big\{ f\in C([0,1],M_{2^\infty}\otimes M_{3^\infty}) : f(0)\in M_{2^\infty}\otimes 1, f(1)\in 1\otimes M_{3^\infty} \big\}.
\]
By \cite[Theorem~3.4]{RorWin10ZRevisited}, $\mathcal{Z}$ is an inductive limit of a sequence of \ca{s} each isomorphic to $Z_{2^\infty,3^\infty}$.
Hence, $A\otimes\mathcal{Z}$ is isomorphic to an inductive limit of \ca{s} isomorphic to $A\otimes Z_{2^\infty,3^\infty}$.
Using that the generator rank behaves well with respect to inductive limits (\autoref{prp:permanence}), we get
\[
\gr(A\otimes\mathcal{Z})
\leq \liminf_{n\to\infty} \gr(A\otimes Z_{2^\infty,3^\infty})
= \gr(A\otimes Z_{2^\infty,3^\infty}).
\]

It thus suffices to verify $\gr(A\otimes Z_{2^\infty,3^\infty})\leq 8$.
Set
\[
I := \big\{ f\in Z_{2^\infty,3^\infty} : f(0)=f(1)=0 \big\}.
\]
Then $I$ is a closed, two-sided ideal in $Z_{2^\infty,3^\infty}$ and $Z_{2^\infty,3^\infty}/I\cong M_{2^\infty}\oplus M_{3^\infty}$.
Since $Z_{2^\infty,3^\infty}$ is nuclear, we obtain a short exact sequence
\[
0 \to A\otimes I \to A\otimes Z_{2^\infty,3^\infty} \to A\otimes(M_{2^\infty}\oplus M_{3^\infty}) \to 0.
\]
Note that $A\otimes I$ is isomorphic to a closed, two-sided ideal in $A\otimes M_{6^\infty}\otimes C([0,1])$.

Since real rank zero is preserved by passing to matrix algebras and inductive limits, we get $\rr(A\otimes M_{k^\infty})=0$ for $k=2,3,6$.
By \autoref{prp:UHFabs}, we obtain $\gr(A\otimes M_{k^\infty})\leq 1$ for $k=2,3,6$.
Using that the generator rank does not increase when passing to closed, two-sided ideals (\autoref{prp:permanence}) at the first step, and \autoref{prp:tensI} for $A\otimes M_{6^\infty}$ at the second step, we get
\[
\gr(A\otimes I)\leq \gr\big( A\otimes M_{6^\infty}\otimes C([0,1]) \big) \leq 6.
\]

By \autoref{prp:gr_sum_rr0}, we have
\[
\gr\big( A\otimes(M_{2^\infty}\oplus M_{3^\infty}) \big)
= \max\big\{ \gr(A\otimes M_{2^\infty}), \gr(A\otimes M_{3^\infty}) \big\}
\leq 1.
\]

Applying the estimate for the generator rank of an extension (\autoref{prp:permanence}), we get
\[
\gr( A\otimes Z_{2^\infty,3^\infty} )
\leq \gr(A\otimes I)+\gr(A\otimes(M_{2^\infty}\oplus M_{3^\infty}) + 1
\leq 8,
\]
as desired.
\end{proof}

\section{Establishing generator rank one}
\label{sec:reduceGr}

In this section we prove our main result:
separable, $\mathcal{Z}$-stable \ca{s} of real rank zero have generator rank one;
see \autoref{prp:mainThm}.
We deduce some interesting corollaries, most importantly that every classifiable, simple, nuclear \ca{} has generator rank one;
see \autoref{prp:classifiable-gr1}.

Recall that the dimension-drop algebra $Z_{2,3}$ is defined as
\[
Z_{2,3} := \big\{ f\in C([0,1],M_2\otimes M_3 : f(0)\in M_2\otimes 1, f(1)\in 1\otimes M_3 \big\}.
\]
Below, we always view $Z_{2,3}$ as the subalgebra of $C([0,1),M_6)$ given by the next result.

\begin{lma}
The dimension-drop algebra $Z_{2,3}$ is isomorphic to the subalgebra of $C([0,1],M_6)$ consisting of the continuous functions $f\colon [0,1]\to M_6$ that satisfy
\[
f(0)=\left(
\begin{array}{cccccc}
\alpha_{11}\!\!\! & \alpha_{12}\!\!\! \\
\alpha_{21}\!\!\! & \alpha_{22}\!\!\! \\
& & \alpha_{11}\!\!\! & \alpha_{12}\!\!\! \\
& & \alpha_{21}\!\!\! & \alpha_{22}\!\!\! \\
& & & & \alpha_{11}\!\!\! & \alpha_{12} \\
& & & & \alpha_{21}\!\!\! & \alpha_{22} \\
\end{array}\right),\
f(1)=\left(
\begin{array}{cccccc}
\beta_{11} \!\!\! & \beta_{12} \!\!\! & \beta_{13} \!\!\! \\
\beta_{21} \!\!\! & \beta_{22} \!\!\! & \beta_{23} \!\!\! \\
\beta_{31} \!\!\! & \beta_{32} \!\!\! & \beta_{33} \!\!\! \\
& & & \beta_{33} \!\!\! & \beta_{31} \!\!\! & \beta_{32} \\
& & & \beta_{13} \!\!\! & \beta_{11} \!\!\! & \beta_{12} \\
& & & \beta_{23} \!\!\! & \beta_{21} \!\!\! & \beta_{22} \\
\end{array}\right),
\]
for some $\alpha_{jk},\beta_{jk}\in\CC$.
\end{lma}
\begin{proof}
Using an identification of $M_2\otimes M_3$ with $M_6$, we naturally view $Z_{2,3}$ as a subalgebra of $C([0,1],M_6)$.
Let $u\in M_6$ be the permutation matrix given as:
\[
u := 
\left(
\begin{array}{cccccc}
1 & 0 & 0 & 0 & 0 & 0 \\
0 & 1 & 0 & 0 & 0 & 0 \\
0 & 0 & 1 & 0 & 0 & 0 \\
0 & 0 & 0 & 0 & 0 & 1 \\
0 & 0 & 0 & 1 & 0 & 0 \\
0 & 0 & 0 & 0 & 1 & 0 \\
\end{array}\right).
\]
Let $t\mapsto v_t$ be a continuous path of unitaries in $M_6$ with $v_0=1$ and $v_1=u$.
Then $v$ is a unitary in $C([0,1],M_6)$ that conjguates $Z_{2,3}\subseteq C([0,1],M_6)$ onto the subalgebra  of functions described in the statement.
\end{proof}

\begin{lma}
\label{prp:construction}
Let $A$ be a unital, separable \ca{} of real rank zero, and let $n\in\NN$ such that $\gr(A)\leq n+2$.
Let $x_0,\ldots,x_{n+1}\in A_\sa$, $\varepsilon>0$ and $z\in A$.
Then there exist $y_0,\ldots,y_{n+1}\in (A\otimes Z_{2,3})_\sa$ such that
\[
\left\| y_j - (x_j\otimes 1) \right\| < \varepsilon\ \text{ for } j=0,\ldots,n+1, \andSep
z\otimes 1 \in_\varepsilon C^*(y_0,\ldots,y_{n+1}).
\]
\end{lma}
\begin{proof}
Consider the $n+3$ elements $\tfrac{\varepsilon}{2},x_0,\ldots,x_{n+1}\in A_\sa$.
Using that $A$ is unital, separable with $\gr(A)\leq n+2$, apply \autoref{prp:grCharSep} to obtain $a,x_0',\ldots,x_{n+1}'\in A_\sa$ such that
\[
\left\| a-\frac{\varepsilon}{2} \right\| < \frac{\varepsilon}{2}, \
\left\| x_j' - x_j \right\| < \varepsilon\ \text{ for } j=0,\ldots,n+1, \ \text{and}\
A=C^*(a,x_0',\ldots,x_{n+1}').
\]
Note that $a$ is positive and invertible.
To simplify notation, set $b:=x_0'$ and $c:=x_1'$.
Choose a polynomial $p$ in noncommuting variables such that
\[
\left\| z-p(a,b,c,x_2',\ldots,x_{n+1}') \right\| < \varepsilon.
\]
Since $p$ is continuous as a map $A^{n+3}_\sa\to A$, we can choose $\delta>0$ such that every $\bar{c}\in A_\sa$ with $\|\bar{c}-c\|<\delta$ satisfies
\[
\left\| z-p(a,b,\bar{c},x_2',\ldots,x_{n+1}') \right\| < \varepsilon.
\]

Since $A$ has real rank zero, we can choose an invertible element $c'\in A_\sa$ with finite spectrum $\sigma(c')$ and such that $\|c'-c\|<\tfrac{\delta}{2}$.
Then $c'=\sum_{j\in J}\lambda_j q_j$ for some finite set $J$, some pairwise disjoint real numbers $\lambda_j$ (the eigenvalues of $c'$), and some pairwise orthogonal projections $q_j$ that sum to $1_A$.

Choose $\mu>0$ such that any two points in $\{0\}\cup\{\lambda_j:j\in J\}$ have distance strictly larger than $4\mu$.
We may assume that $2\mu<\varepsilon$ and $2\mu<\tfrac{\delta}{2}$.

Next, we define auxiliary functions $\alpha^{(k)}\colon[0,1]\to\RR$ for $k=1,\ldots,6$ by
\begin{align*}
\alpha^{(1)}(t):= 1+t(1-t), \quad
\alpha^{(5)}(t):= 1, \quad
\alpha^{(3)}(t):=1-t,
\end{align*}
and
\[
\alpha^{(6)}(t):=-\alpha^{(1)}(t), \quad
\alpha^{(2)}(t):=-\alpha^{(5)}(t), \quad
\alpha^{(4)}(t):=-\alpha^{(3)}(t), \quad
\]
for $t\in[0,1]$.
The functions are shown in the following picture:
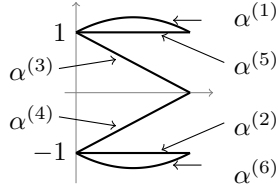
\[
\makebox{
\begin{tikzpicture}[xscale=1.5, yscale=0.8]
\draw [help lines, ->] (-0.1,0) -- (1.2,0);
\draw [help lines, ->] (0,-1.5) -- (0,1.5);
\draw [thick, domain=0:1, samples=500] plot (\x, {1+\x-\x^2});
\draw [thick] (0,1) -- (1,1);
\draw [thick] (0,1) -- (1,0);
\draw [thick] (0,-1) -- (1,0);
\draw [thick] (0,-1) -- (1,-1);
\draw [thick, domain=0:1, samples=500] plot (\x, {-1-\x+\x^2});
\node [left] at (0,1) {$1$};
\node [left] at (0,-1) {$-1$};
\draw [->]  (1.1,1.2) -- (0.85,1.2);
\node [right] at (1.3,1.3) {$\alpha^{(1)}$};
\draw [->]  (1.3,0.5) -- (0.75,0.95);
\node [right] at (1.3,0.5) {$\alpha^{(5)}$};
\draw [->] (-0.15,0.4) -- (0.35,0.55);
\node [left] at (-0.1,0.4) {$\alpha^{(3)}$};
\draw [->] (-0.15,-0.4) -- (0.35,-0.55);
\node [left] at (-0.1,-0.4) {$\alpha^{(4)}$};
\draw [->]  (1.3,-0.5) -- (0.75,-0.95);
\node [right] at (1.3,-0.5) {$\alpha^{(2)}$};
\draw [->]  (1.1,-1.2) -- (0.85,-1.2);
\node [right] at (1.3,-1.3) {$\alpha^{(6)}$};
\end{tikzpicture}
}
\]

We note the following properties:
\begin{enumerate}
\item[(a)]
$\alpha^{(k)}$ is continuous with $\|\alpha^{(k)}\|_\infty<2$, for $k=1,\ldots,6$;
\item[(b)]
$\alpha^{(1)}(0)=\alpha^{(5)}(0)=\alpha^{(3)}(0)$ and $\alpha^{(4)}(0)=\alpha^{(2)}(0)=\alpha^{(6)}(0)$;
\item[(c)]
$\alpha^{(1)}(1)=\alpha^{(5)}(1)$, $\alpha^{(3)}(1)=\alpha^{(4)}(1)$, and $\alpha^{(2)}(1)=\alpha^{(6)}(1)$;
\item[(d)]
$\alpha^{(6)}(t)<\alpha^{(2)}(t)<\alpha^{(4)}(t)<\alpha^{(3)}(t)<\alpha^{(5)}(t)<\alpha^{(1)}(t)$ for each $t\in(0,1)$.
\end{enumerate}

For each $k=1,\ldots,6$, we define $f_{kk}\colon[0,1]\to A$ by
\begin{align*}
f_{kk}(t) := \sum_{j\in J} \big(\lambda_j+\mu \alpha^{(k)}(t)\big) q_j
\end{align*}
for $t\in[0,1]$.
We let $e_{kl}\in M_6$ denote the matrix units, for $k,l=1,\ldots,6$.
Then define $f,g\colon[0,1]\to A\otimes M_6$ by
\begin{align*}
f(t) &:= \sum_{k=1}^6 f_{kk}(t)\otimes e_{kk}, \\
g(t) &:= b\otimes 1 
+ a\otimes(e_{12}+e_{21}+e_{56}+e_{65}) 
+ \mu t (e_{23}+e_{32}+e_{46}+e_{64}) \\
&\quad + (1-t)a(e_{34}+e_{43}).
\end{align*}
for $t\in[0,1]$.

This means, that $f(t)$ and $g(t)$ have the following matrix form:
\begin{align*}
f(t) &:=\left( \begin{array}{ccc}
f_{11} \\
& \ddots \\
& & f_{66}\\
\end{array} \right),\
g(t) &:=\left( \begin{array}{cc|cc|cc}
b & a & & & \\
a & b & \mu t & & \\
\hline 
& \mu t & b & (1-t)a & & \\
& & (1-t)a & b & & \mu t \\
\hline 
& & & & b & a \\
& & & \mu t & a & b\\
\end{array} \right).
\end{align*}

We view elements in $A\otimes Z_{2,3}$ as continuous functions $[0,1]\to A\otimes M_6$.
By~(a), $f$ is continuous.
Using~(b) and~(c), we deduce that $f$ belongs to $A\otimes Z_{2,3}$.
We also have $g\in A\otimes Z_{2,3}$.
Set
\[
B := C^*(f,g,x_2'\otimes 1,\ldots,x_{n+1}'\otimes 1)\subseteq A\otimes Z_{2,3}.
\]
For each $t\in[0,1]$, we let $B(t)\subseteq A\otimes M_6$ be the image of $B$ under the evaluation map $A\otimes Z_{2,3}\to A\otimes M_6$, $h\mapsto h(t)$.
We use $e_{kl}$ to denote the matrix units in $M_6$.

\emph{Claim~1a: Let $t\in(0,1)$. Then $1\otimes M_6\subseteq B(t)$.}
To prove the claim, note that the spectrum of $f_{kk}(t)$ is
\[
\sigma(f_{kk}(t)) = \big\{ \lambda_j+\alpha_j^{(k)}(t) : j\in J \big\},
\]
for $k=1,\ldots,6$.
Using~(d), we obtain that $\sigma(f_{kk}(t))$ and $\sigma(f_{ll}(t))$ are disjoint whenever $k\neq l$.

Given $k\in\{1,\ldots,6\}$, let $h_k\colon\RR\to[0,1]$ be a continuous function that takes the value $1$ on $\sigma(f_{kk}(t))$, and that takes the value $0$ on $\{0\}\cup\bigcup_{l\neq k}\sigma(f_{ll}(t))$.
Then $1\otimes e_{kk}=h_k(f(t))\in B(t)$.
Thus, $B(t)$ contains the diagonal matrix units of $1\otimes M_6$.

It follows that $B(t)$ contains
\[
1\otimes e_{23}
= \frac{1}{\mu t} (1\otimes e_{22})g(t)(1\otimes e_{33}), \andSep
1\otimes e_{46}
= \frac{1}{\mu t} (1\otimes e_{44})g(t)(1\otimes e_{66}).
\]

We have $a\otimes e_{12} = (1\otimes e_{11})g(t)(1\otimes e_{22}) \in B(t)$, and thus
\[
a^2\otimes e_{22} = (a\otimes e_{12})^*(a\otimes e_{12}) \in B(t).
\]
As in the proof of \autoref{prp:UHFabs}, we use that $a$ is positive and invertible, to deduce $a^{-1}\otimes e_{22}\in B(t)$, and thus
\[
1\otimes e_{12} = (a\otimes e_{12})(a^{-1}\otimes e_{22})\in B(t).
\]
Analogously, we get that $B(t)$ contains $1\otimes e_{34}$ and $1\otimes e_{56}$.
It follows that $B(t)$ contains $1\otimes e_{kl}$ for every $k,l\in\{1,\ldots,6\}$, and so $1\otimes M_6\subseteq B(t)$.
Similarly, one proves:

\emph{Claim~1b: We have
\[
1\otimes(e_{kl}+e_{k+2,l+2}+e_{k+4,l+4})\in B(0), \text{ for } k,l\in\{1,2\}.
\]
}

\emph{Claim~1c: We have $1\otimes(e_{33}+e_{44})\in B(1)$ and
\[
1\otimes(e_{kl}+e_{k+4,l+4})\in B(1), \text{ for } k,l\in\{1,2\}.
\]
and
\[
1\otimes(e_{l3}+e_{l+4,4}), 1\otimes(e_{3l}+e_{4,l+4}), \in B(1), \text{ for } l\in\{1,2\}.
\]
}

\emph{Claim~2: Let $t\in[0,1]$. Then $z\otimes 1_{M_6} \in_\varepsilon B(t)$.}
First, we assume that $t\in(0,1)$.
By Claim~1a, we have $1\otimes e_{11},1\otimes e_{21}\in B(t)$, and thus
\[
a\otimes e_{11} = (1\otimes e_{11})g(t)(1\otimes e_{21})\in B(t).
\]
Analogously, we obtain
\[
b\otimes e_{11},\ f_{11}(t)\otimes e_{11},\ x_2'\otimes e_{11},\ldots,x_{n+1}'\otimes e_{11}\in B(t).
\]
We have
\[
\| f_{11}(t)-c \|
\leq \| f_{11}(t)-c' \| + \| c'-c \|
< 2\mu +\frac{\delta}{2}
\leq \delta.
\]
By choice of $\delta$, we get
\[
\left\| z-p(a,b,f_{11}(t),x_2',\ldots,x_{n+1}') \right\| < \varepsilon.
\]
and thus
\[
z\otimes e_{11} \in_\varepsilon B(t).
\]
It follows that $z\otimes e_{kk} \in_\varepsilon B(t)$ for each $k$, and consequently $z\otimes 1 \in_\varepsilon B(t)$.

Next, we consider the case $t=0$.
Set $\tilde{e}_{kl}:=e_{kl}+e_{k+2,l+2}+e_{k+4,l+4}\in M_6$ for $k,l\in\{1,2\}$.
By Claim~1b, we have $1\otimes\tilde{e}_{kl}\in B(0)$ for each $k,l$, and thus
\[
a\otimes \tilde{e}_{11}
= (1\otimes\tilde{e}_{11}) g(0) (1\otimes\tilde{e}_{21}) \in B(0).
\]
Analogously, we obtain
\[
b\otimes \tilde{e}_{11},\ f_{11}(0)\otimes \tilde{e}_{11},\ x_2'\otimes \tilde{e}_{11},\ldots,x_{n+1}'\otimes \tilde{e}_{11}\in B(0).
\]
Arguing as in the proof of Claim~2a, we get $z\otimes\tilde{e}_{11}\in_\varepsilon B(0)$.
It follows that $z\otimes\tilde{e}_{22}\in_\varepsilon B(0)$, and consequently $z\otimes 1 \in_\varepsilon B(0)$.

Similarly, one proves $z\otimes 1 \in_\varepsilon B(1)$.

\emph{Claim~3: Let $s,t\in[0,1]$ with $s\neq t$. Then there exists $d\in B$ such that $d(s)=0$ and $d(t)=1$.}
We first consider the case $s<t$.
By choice of $\mu$, the intervals $[\lambda_j-2\mu,\lambda_j+2\mu]$ are pairwise disjoint for $j\in J$.
We may therefore choose a continuous function $h\colon\RR\to[0,1]$ that takes the value $0$ on 
\[
S:=\bigcup_{j\in J} \big( [\lambda_j-2\mu,\lambda_j-(1-s)\mu] \cup [\lambda_j+(1-s)\mu,\lambda_j+2\mu]).
\]
and that takes the value $1$ on
\[
T:=\bigcup_{j\in J} [\lambda_j-(1-t)\mu,\lambda_j+(1-t)\mu],
\]
Note that $T$ consists of the real numbers that have distance at most $(1-t)\mu$ to some $\lambda_j$.
For the purposes below, one could consider $S$ as the real numbers that have distance at least $(1-s)\mu$ to each $\lambda_j$.

Set $c:=h(f)\in B$, the element obtained by applying functional calculus for $h$ to $f$.
Since $f=\diag(f_{11},\ldots,f_{66})$, we have $c=\diag(h(f_{11}),\ldots,h(f_{66})$.

Let $r\in[0,1]$.
We have
\[
\sigma(f_{kk}(r))=\{\lambda_j+\mu\alpha^{(k)}(r):j\in J\}.
\]
If $k\in\{1,5,2,6\}$, then $|\alpha^{(k)}(r)|\geq 1$, and thus
\[
|\lambda_j-(\lambda_j+\mu\alpha^{(k)}(r))|\geq \mu \geq (1-s)\mu,
\]
whence
\[
\sigma(f_{kk}(r))\subseteq S.
\]
It follows that 
\[
h(f_{11})=h(f_{55})=h(f_{22})=h(f_{66})=0.
\]
For $k\in\{3,4\}$, we have
\[
|\lambda_j-(\lambda_j+\mu\alpha^{(k)}(s))| = \mu|1-s|, \andSep
|\lambda_j-(\lambda_j+\mu\alpha^{(k)}(t))| = \mu|1-t|,
\]
whence 
\[
\sigma(f_{kk}(s))\subseteq S, \andSep 
\sigma(f_{kk}(t))\subseteq T.
\]
It follows that 
\[
h(f_{33})(s)=h(f_{44})(s)=0, \andSep
h(f_{33})(t)=h(f_{44})(t)=1.
\]
In conclusion, we have
\[
c(s)=0, \andSep 
c(t)=1\otimes(e_{33}+e_{44}).
\]

If $t<1$, then by Claim~1a there exist $b_{kl}\in B$ such that $b_{kl}(t)=1\otimes e_{kl}$ for $k,l\in\{1,\ldots,6\}$.
Then the following element has the desired properties:
\[
d := c + b_{13}cb_{31} + b_{23}cb_{32} + b_{53}cb_{35} + b_{63}cb_{36}.
\]

If $t=1$, then by Claim~1c there exist $b_{k3},b_{3l}\in B$ such that $b_{k3}(1)=1\otimes (e_{k3}+e_{k+4,4})$ and $b_{3l}(1)=1\otimes (e_{3l}+e_{4,l+4})$ for $k,l\in\{1,2\}$.
Then the following element has the desired properties:
\[
d := c + b_{13}cb_{31} + b_{23}cb_{32}.
\]

Next, let us indicate how to proceed in the case $s>t$.
Let $h\colon\RR\to[0,1]$ be a continuous function that takes the value $0$ on 
\[
S:=\bigcup_{j\in J} \big( [\lambda_j-2\mu, \lambda_j-\mu] \cup [\lambda_j-(1-s)\mu,\lambda_j+(1-s)\mu] \cup [\lambda_j+\mu,\lambda_j+2\mu]).
\]
and that takes the value $1$ on
\[
T:=\bigcup_{j\in J} \{\lambda_j-(1-t)\mu,\lambda_j+(1-t)\mu\}.
\]
The element $c:=h(f)\in B$ satisfies $c(s)=0$ and $c(t)\neq 0$.
Similar as in the case $s<t$, one then constructs $d\in B$ such that $d(s)=0$ and $d(t)=1$, which proves the claim.

Claim~3 verifies condition~(b) in \autoref{prp:genCenter}, whence $B$ contains $C([0,1])$.
Thus, $B$ is a $C([0,1])$-subalgebra of $A\otimes Z_{2,3}$.
If $\pi_t$ denotes the quotient map from $A\otimes Z_{2,3}$ to the fiber at $t$, then Claim~2 shows that $\pi_t(z\otimes 1)\in_\varepsilon\pi_t(B)=B(t)$ for every $t\in[0,1]$.
By \cite[Lemma~2.1]{Dad09CtsFieldsOverFD}, we get $z\otimes 1\in_\varepsilon B$, which proves the result.
\end{proof}

\begin{thm}
\label{prp:mainThm}
Let $A$ be a separable, $\mathcal{Z}$-stable \ca{} of real rank zero.
Then~$A$ has generator rank one.
In particular, a generic element of $A$ is a generator.
\end{thm}
\begin{proof}
We first prove the theorem under the additional assumption that $A$ is unital.
In this case, we have $\gr(A)\leq 8$ by \autoref{prp:gr8}.
Next, we successively reduce the upper bound for $\gr(A)$.

\emph{Claim: 
Let $n\in\NN$ such that $\gr(A)\leq n+2$.
Then $\gr(A)\leq n+1$.
}
To prove the claim, we verify condition~(3) of \autoref{prp:genDabsorbing}. 
Let $x_0,\ldots,x_{n+1}\in A_\sa$, $\varepsilon>0$ and $z\in A$.
We need to find $y_0,\ldots,y_{n+1}\in(A\otimes\mathcal{Z})_\sa$ such that
\[
\left\| y_j - (x_j\otimes 1) \right\| < \varepsilon\ \text{ for } j=0,\ldots,n+1, \andSep
z\otimes 1 \in_\varepsilon C^*(y_0,\ldots,y_{n+1}).
\]
By identifying $Z_{2,3}$ with a unital sub-\ca{} of $\mathcal{Z}$, elements $y_0,\ldots,y_{n+1}$ with the desired properties are provided by \autoref{prp:construction}.

Applying the claim seven times, we obtain that $\gr(A)\leq 1$.

If $A$ is nonunital, we use that $A$ is separable and has real rank zero to choose an increasing approximate unit $(p_n)_n$ of projections in $A$;
see \cite[Proposition~2.9]{BroPed91CAlgRR0}.
For each~$n$, we consider the unital corner $A_n:=p_nAp_n$.
By \cite[Corollary~2.8]{BroPed91CAlgRR0}, real rank zero passes to hereditary sub-\ca{s}.
By \cite[Corollary~3.1]{TomWin07ssa}, $\mathcal{Z}$-stability passes to hereditary sub-\ca{s}.
Thus, $A_n$ is a unital, separable, $\mathcal{Z}$-stable \ca{} of real rank zero, and thus $\gr(A_n)\leq 1$.
By \autoref{prp:permanence}, we get
\[
\gr(A)\leq\liminf_n \gr(A_n) = 1.
\]
Since $A$ is noncommutative (if nonzero), we have $\gr(A)\neq 0$ by \autoref{prp:gr0}, and so $\gr(A)=1$.
Since $A$ is separable and has real rank zero, $\gr(A)=1$ means that generators in $A$ are a dense $G_\delta$-subset;
see \autoref{rmk:gr1}
\end{proof}

\begin{rmk}
Using the methods developed in \cite[Section~4]{Thi12arX:GenRnk}, one can remove the assumption of separability in \autoref{prp:mainThm}:
Every $\mathcal{Z}$-stable \ca{} of real rank zero has generator rank one.
They key point is that for every $\mathcal{Z}$-stable \ca{} $A$ and every separable sub-\ca{} $B_0\subseteq A$, there exists a separable, $\mathcal{Z}$-stable sub-\ca{} $B\subseteq A$ with $B_0\subseteq B$.
Similar methods are used in the proof of \autoref{prp:nuclPI} below.
\end{rmk}

For the definition and the basic properties of pure infiniteness for nonsimple \ca{s}, we refer to \cite{KirRor00PureInf}.

\begin{cor}
\label{prp:nuclPI}
Every nuclear, purely infinite \ca{} of real rank zero has generator rank one.
\end{cor}
\begin{proof}
Let $A$ be a nuclear, purely infinite \ca{} of real rank zero.
As in \cite[Paragraph~4.1]{Thi12arX:GenRnk}, we let $\SubSep(A)$ denote the collection of separable sub-\ca{s} of $A$; a subset $\mathcal{S}\subseteq\SubSep(A)$ is \emph{$\sigma$-complete} if $\overline{\bigcup\mathcal{T}}\in\mathcal{S}$ for every countable, directed subset $\mathcal{T}\subseteq\mathcal{S}$;
a subset $\mathcal{S}\subseteq\SubSep(A)$ is \emph{cofinal} if for every $B_0\in\SubSep(A)$ there is $B\in\mathcal{S}$ with $B_0\subseteq B$.

We let $\mathcal{S}_{\mathrm{nuc}}$, $\mathcal{S}_{\mathrm{pi}}$ and $\mathcal{S}_{\mathrm{rr0}}$ denote the sets of separable sub-\ca{s} of $A$ that are nuclear, purely infinite and have real rank zero, respectively.
It follows from Paragraph~II.9.6.5 and Proposition~IV.3.1.9 in \cite{Bla06OpAlgs} that $\mathcal{S}_{\mathrm{nuc}}$ is $\sigma$-complete and cofinal.
Using Proposition~4.18 and Corollary~4.22 in \cite{KirRor00PureInf}, it follows that $\mathcal{S}_{\mathrm{pi}}$ is $\sigma$-complete and cofinal.
Lastly, it was noted at the end of \cite[Paragraph~4.1]{Thi12arX:GenRnk} that real rank zero satisfies the `L\"{o}wenheim-Skolem condition', which means that $\mathcal{S}_{\mathrm{rr0}}$ is $\sigma$-complete and cofinal.
Set
\[
\mathcal{S}:= \mathcal{S}_{\mathrm{nuc}}\cap\mathcal{S}_{\mathrm{pi}}\cap\mathcal{S}_{\mathrm{rr0}}.
\]
It is well-known that the intersection of countably many $\sigma$-complete, cofinal subsets is again $\sigma$-complete and cofinal, whence $\mathcal{S}$ is $\sigma$-complete and cofinal.

Let $B\in\mathcal{S}$.
Then $B$ is a separable, nuclear, purely infinite \ca{} of real rank zero.
It follows from \cite[Theorem~9.1]{KirRor02InfNonSimpleCalgAbsOInfty} that $B$ is $\mathcal{O}_\infty$-stable, and thus $\mathcal{Z}$-stable.
By \autoref{prp:mainThm}, we have $\gr(B)=1$.
Since $A$ is the inductive limit of the system $\mathcal{S}$ (indexed over itself), we obtain $\gr(A)\leq 1$ by \autoref{prp:permanence}.
Since purely infinite \ca{s} are by definition noncommutative, we we have $\gr(A)\neq 0$ by \autoref{prp:gr0}, and so $\gr(A)=1$.
\end{proof}

Recall that a \emph{Kirchberg algebra} is a separable, simple, nuclear, purely infinite \ca.
By Zhang's theorem, \cite[Proposition~V.3.2.12, p.454]{Bla06OpAlgs}, every simple, purely infinite \ca{} has real rank zero.
Thus, Kirchberg algebras have real rank zero.
Applying \autoref{prp:nuclPI}, we obtain:

\begin{cor}
\label{prp:KirchbergAlg}
Every Kirchberg algebra has generator rank one.
\end{cor}

Let us say that a simple, nuclear \ca{} is \emph{classifiable} if it is unital, separable, $\mathcal{Z}$-stable and satisfies the Universal Coefficient Theorem (UCT).
By the recent breakthrough in the Elliott classification program, \cite{GonLinNiu15arX:mainClassification, EllGonLinNiu15arX:classFinDR2, TikWhiWin17QDNuclear}, two simple, nuclear, classifiable \ca{s} are isomorphic if and only if their Elliott invariants ($K$-theoretic and tracial data) are isomorphic.

\begin{cor}
\label{prp:classifiable-gr1}
Let $A$ be a unital, separable, simple, nuclear, $\mathcal{Z}$-stable \ca{} satisfying the UCT.
Then $A$ has generator rank one.
In particular, a generic element in $A$ is a generator.
\end{cor}
\begin{proof}
By \cite[Theorem~4.1.10]{Ror02Classification}, $A$ is either stably finite or purely infinite.
In the second case, $A$ is a Kirchberg algebra and we obtain $\gr(A)=1$ by \autoref{prp:KirchbergAlg}.
(The purely infinite case does not require the UCT.)

In the first case, it follows from \cite[Theorem~6.2(iii)]{TikWhiWin17QDNuclear} that $A$ is an approximately subhomogeneous (ASH) algebra.
By \cite[Theorem~5.10]{Thi20arX:grSubhom}, every $\mathcal{Z}$-stable ASH-algebra has generator rank one.
Thus, we have $\gr(A)=1$ in either case.

Since $A$ is unital and separable, $\gr(A)=1$ means that generators in $A$ are a dense $G_\delta$-subset;
see \autoref{rmk:gr1}.
\end{proof}

\begin{rmks}
(1)
It seems likely that the proof of the main theorem can be generalized to show the following:
If $A$ is a unital, separable, $\mathcal{Z}$-stable \ca{} such that $A\otimes M_{2^\infty}$ and $A\otimes M_{3^\infty}$ have real rank zero, then $A$ has generator rank one.

(2)
Let $A$ be a unital, separable, $\mathcal{Z}$-stable \ca{}.
By \cite[Theorem~3.8]{ThiWin14GenZStableCa}, $A$ is singly generated.
Our results show that under additional assumptions, $A$ even contains a dense set of generators.
It is reasonable to expect that every $\mathcal{Z}$-stable \ca{} has generator rank one.
However, by \cite[Proposition~3.10]{Thi12arX:GenRnk}, the real rank is a lower bound for the generator rank, and it is not known that every $\mathcal{Z}$-stable \ca{} has real rank at most one.

Note that every unital, separable, \emph{simple}, $\mathcal{Z}$-stable \ca{} has real rank at most one:
It is either purely infinite and then has real rank zero;
or it is stably finite and thus has stable rank one by \cite[Theorem~6.7]{Ror04StableRealRankZ}, which entails real rank at most one.
Therefore, the following question has no obvious obstruction:
\end{rmks}

\begin{qst}
Does every unital, separable, \emph{simple}, $\mathcal{Z}$-stable \ca{} have generator rank one?
\end{qst}


\providecommand{\href}[2]{#2}

\end{document}